\documentclass[a4paper,11pt,reqno]{article}
\usepackage{amsmath,amssymb,eucal,amsthm,color,mathrsfs}
\newcommand{\B}[2]{\mathbb{B}^{(#1)}_{#2} }% k-th Bernoulli #
\newcommand{\C}[2]{C^{(#1)}_{#2} }% k-th Bernoulli #
\newcommand{\st}[2]{\genfrac{\{}{\}}{0pt}{}{#1}{#2}} % Stirling # of the second kind

\DeclareMathOperator{\Li}{Li} % polylogarithm

\newtheorem{theorem}{Theorem}[section]

\newtheorem{corollary}[theorem]{Corollary}

\theoremstyle{definition}

\newtheorem{remark}[theorem]{Remark}

\begin{document}

\title{On congruence properties of poly-Bernoulli numbers with negative upper-indices}

\author{Yasuo Ohno and Mika Sakata}

\date{}

\maketitle

\begin{abstract}
For any integer $k$, M.~Kaneko (\cite{K1997}) defined $k$-th poly-Bernoulli numbers as a kind of generalization of classical Bernoulli numbers using $k$-th polylogarithm. 
In case when $k$ is positive, $k$-th poly-Bernoulli numbers is a sequence of rational numbers as same as classical  Bernoulli numbers. 
On the other hand, in case when $k$ is negative, it is a sequence of positive integers, and many combinatoric and number theoretic properties has been investigated. 
In the present paper, the negative case is treated, and their congruence and $p$-adic properties are discussed. 
Beside of them, 
application of the results to obtain a congruence property for the number of lonesum matrices is also mentioned.
\end{abstract}

%%%%%%%%%%%%%%%%%%%%%%%%%%%%%%%%%%%
\section{Introduction}

Poly-Bernoulli numbers $\{\B{k}{n}\}_{n\geq0}$ is a sequence defined, in \cite{K1997}, for any integer $k$ by the generating series 
\begin{align*}
\frac{\Li_{k}(1-e^{-t})}{1-e^{-t}}=\sum^{\infty}_{n=0}\B{k}{n}\frac{t^n}{n!} 
\end{align*}
where $\Li_{k}(t)=\sum^{\infty}_{n=1}\frac{t^{n}}{n^{k}}$. 
In addition, Keneko defined another sequence $\{\C{k}{n}\}_{n\geq0}$ by 
\begin{align*}
\frac{\Li_{k}(1-e^{-t})}{e^{t}-1}=\sum^{\infty}_{n=0}\C{k}{n}\frac{t^n}{n!} 
\end{align*}
in \cite{AKzeta}, it will be called poly-Bernoulli numbers of ``type-$C$" in the current paper. 
It is well-known that $\B{1}{n}$ (resp. $\C{1}{n}$) is the classical Bernoulli numbers with 
$\B{1}{1}=\frac12$ (resp. $\C{1}{1}=-\frac12$). 
Moreover, for any integer $k$, the numbers $\B{k}{n}$ and $\C{k}{n}$ are related with each other as follows:  
\[ \B{k}{n}=\sum_{m=0}^{n}\binom{n}{m}\C{k}{m}, \quad \C{k}{n}=\sum_{m=0}^{n}(-1)^{n-m}\binom{n}{m}\B{k}{m}\quad{\rm and}\quad \B{k}{n}=\C{k}{n}+\C{k-1}{n-1}. \]

An explicit formula for each poly-Bernoulli numbers is given 
in \cite[Theorem 1]{K1997} and \cite[Theorem 1]{K2010}
as follows: 
\begin{align}\label{explicit formula}
\B{k}{n}=(-1)^{n}\sum^{n}_{m=0}\frac{(-1)^{m}m!\st{n}{m}}{(m+1)^{k}}
\end{align}
and
\begin{align}\label{explicit formula2}
\C{k}{n}=(-1)^{n}\sum^{n}_{m=0}\frac{(-1)^{m}m!\st{n+1}{m+1}}{(m+1)^{k}},
\end{align}
where $\st{n}{m}$ denotes the Stirling number of the second kind, defined by
\begin{align*}
\frac{(e^t-1)^m}{m!}=\sum_{n=m}^{\infty}\st{n}{m}\frac{t^n}{n!}
\end{align*}
or equivalently by
\begin{align}\label{stirling explicit}
\st{n}{m} = \frac{(-1)^m}{m!}\sum_{l=0}^{m}(-1)^l \binom{m}{l}l^{n}, 
\end{align}
for any non-negative integers $n$ and $m$.

Studying the divisibility of the denominators of classical Bernoulli numbers by prime numbers is important and derives the Kummer congruence and $p$-adic theory of $L$ functions.
Poly-Bernoulli numbers with positive upper indices are sequence of rational numbers, and their divisibility by prime numbers in the denominators has been studied (\cite{AIK,AK1999,K1997} et.al.). 
Hoffman (\cite{Hoffman}) showed an interesting congruence between poly-Bernoulli numbers with positive upper indices and finite multiple zeta values. It is also known in \cite{AKzeta} that poly-Bernoulli numbers with positive upper indices appear at the negative integer points of the Arakawa--Kaneko multiple zeta functions,  and its congruence property and a kind of periodicity are also studied by the authors in \cite{OSaka1}.

In case when the upper-index is negative, 
poly-Bernoulli numbers are sequence of positive integers. 
These sequences and each term of them are meaningful objects not only in number theory but also in combinatorics, 
and a number of interesting results have been obtained. 
For example, a beautiful relation named ``duality" are well known.
For any non-negative integers $k$ and $n$, both poly-Bernoulli numbers have each duality formula(\cite[Theorem 2]{K1997}, \cite[Corollary of Theorem 1]{K2010}): 
\begin{align}\label{duality formula}
\B{-k}{n}=\B{-n}{k}, 
\end{align}
and 
\begin{align}\label{duality formula2}
\C{-k-1}{n}=\C{-n-1}{k}. 
\end{align}
Moreover a pretty formula
$$
\sum_{l=0}^{n} (-1)^l \B{-l}{n-l}=0
$$
is obtained in $\cite{AK1999}$ for any positive integer $n$. 
More combinatorially interesting properties are known by many papers including 
\cite{AT, Benyi, Brew, KOY, KST, Launois, OS}.

In the present paper, poly-Bernoulli numbers with negative upper-indices are treated, 
and their congruence and $p$-adic properties are investigated. 
Beside of them, 
application of the results to obtain a congruence property for the number of lonesum matrices, via Brewbaker's work, is also mentioned.

Throughout this paper, the characters $p$ and $p_j$ express prime numbers, 
and $M$ and $N$ express positive integers.
Moreover, $\varphi (M)$ denotes the Euler totient function of $M$, 
which gives the number of positive integers relatively prime to $M$ and less than $M$. 
For example, $\varphi(p^N)=p^{N-1}(p-1)$.

%%%%%%%%%%%%%%%%%%%%%%%%%%%%%%%%%%%
\section{Periodicity}
%%%
In this section, we present our results on the properties of poly-Bernoulli numbers with negative upper indices modulo $M$. 
 
First, we show that poly-Bernoulli numbers with negative upper indices have a period modulo $p^N$, as follows. 

\begin{theorem}\label{period1}
Let $p$ be a prime number, and $k, N$ be positive integers. 
Suppose that integers $n, m \geq N$ satisfies $n\equiv m \pmod{\varphi(p^N)}$. 
Then we have  
\begin{eqnarray*}
\B{-k}{n}\equiv \B{-k}{m}\pmod{p^{N}}. 
\end{eqnarray*}
\end{theorem}

\begin{proof}
Assuming that $n, m \geq N$, using the explicit formula (\ref{explicit formula}), we have 
\begin{eqnarray*}
\B{-n}{k}
&=& (-1)^{k}\sum^{k}_{l=0}(-1)^{l}l!\st{k}{l}(l+1)^{n}\\
&\equiv & (-1)^{k}\sum^{k}_{\substack{l=0\\p\not \hspace{2pt} |\hspace{2pt}l+1}}(-1)^{l}l!\st{k}{l}(l+1)^{n}
\pmod{p^{N}}. 
\end{eqnarray*}
Euler's totient theorem, namely the natural generalization of Fermat's little theorem, 
confirms 
$(l+1)^n\equiv (l+1)^m$ (mod $p^N$) for any non-negative integers $l$ satisfying $p\not |\hspace{2pt}l+1$, 
and it leads us to obtain 
\begin{eqnarray*}
\B{-n}{k}\equiv \B{-m}{k} \pmod{p^{N}}. 
\end{eqnarray*}
Thus we obtain the claim $\B{-k}{n}\equiv \B{-k}{m}\pmod{p^{N}}$, by applying the duality formula (\ref{duality formula}) to the both sides. 
\end{proof}

\begin{remark}
Theorem \ref{period1} originally obtained by both Kitahara \cite{Kitahara} and Ohno--Sakata \cite{OSaka} almost at the meantime independently, 
however the methods are different with each other, i.e. the proof in \cite{Kitahara} uses $p$-adic analytic distribution.
\end{remark}

By Theorem \ref{period1}, we see that poly-Bernoulli numbers with negative upper-index have a period modulo a positive integer $M$. 
\begin{corollary}\label{BcongM}
Let $k$ and $M$ be positive integers, and $M=p_{1}^{e_{1}}p_{2}^{e_{2}} \cdots p_{l}^{e_{l}}$ 
be the prime factorization of $M$, where $p_{1}, p_{2}, \ldots , p_{l}$ are distinct prime divisors, of $M$. 
For any integers $n, m \geq \max \{e_{j}\;|\;1\leq j \leq l\}$ 
satisfying $n\equiv m\pmod{\varphi (M)}$, we have 
\begin{eqnarray*}
\B{-k}{n}\equiv \B{-k}{m}\pmod{M}.
\end{eqnarray*}
\end{corollary}

\begin{proof}
Under the given conditions, $M=p_{1}^{e_{1}}p_{2}^{e_{2}} \cdots p_{l}^{e_{l}}$ 
and $\varphi (M) | (n-m)$, Euler's product formula leads us to 
$\varphi (p_{j}^{e_{j}}) | (n-m)$ for all $j$ satisfying $1\leq j \leq l$, 
and then by Theorem \ref{period1}, we have 
$p_{j}^{e_{j}} | \left(\B{-k}{n}-\B{-k}{m}\right)$, 
for all such $j$. 
Thus we obtain $M | \left(\B{-k}{n}-\B{-k}{m}\right)$. 
\end{proof}

Now, we briefly touch on the combinatorial aspect. 
A matrix whose entries are 0 or 1 and is uniquely determined by its row and column sum vectors
is called ``lonesum". For any integers $k$ and $n$, 
we denote the number of lonesum matrices of size $k\times n$ by $L(k, n)$. 
Corollary \ref{BcongM} together with Brewbaker's relation: 
\begin{align}\label{Brew}
L(k,m)=\B{-k}{m}
\end{align}
(\cite[Theorem 2]{Brew}) immediately derive the period of the number of lonesum matrices, as follows. 

\begin{theorem}
Let $k$ and $M$ be positive integers, and $M=p_{1}^{e_{1}}p_{2}^{e_{2}} \cdots p_{l}^{e_{l}}$ 
be the prime factorization of $M$, where $p_{1}, p_{2}, \ldots , p_{l}$ are distinct prime divisors, of $M$. 
For any integers $n, m \geq \max \{e_{j}\;|\;1\leq j \leq l\}$ 
satisfying $n\equiv m\pmod{\varphi (M)}$, we have 
\[ L(k, n)\equiv L(k, m)\pmod{M}. \]
\end{theorem}

Similar to Theorem \ref{period1}, we get the period for poly-Bernoulli numbers of type-$C$. 
\begin{theorem}\label{period1'}
Let $p$ be a prime number, and $k, N$ be positive integers. 
Suppose that integers $n, m \geq N$ satisfies $n\equiv m \pmod{\varphi(p^N)}$. 
Then we have  
\begin{eqnarray*}
C^{(-k)}_{n}\equiv C^{(-k)}_{m}\pmod{p^{N}}. 
\end{eqnarray*}
\end{theorem}

\begin{proof}
It seems that a reasonable way to prove it would be to give a direct proof, rather than using Theorem \ref{period1}.

Assuming that $n, m \geq N$, using the explicit formula (\ref{explicit formula2}), we have 
\begin{eqnarray*}
C^{(-n-1)}_{k-1}
&=& (-1)^{k-1}\sum^{k-1}_{l=0}(-1)^{l}l!\st{k}{l+1}(l+1)^{n+1}\\
&\equiv & (-1)^{k-1}\sum^{k-1}_{\substack{l=0\\p\not \hspace{2pt} |\hspace{2pt}l+1}}(-1)^{l}l!\st{k}{l+1}(l+1)^{n+1}
\pmod{p^{N}}. 
\end{eqnarray*}
Euler's totient theorem leads us to obtain 
\begin{eqnarray*}
C^{(-n-1)}_{k-1}\equiv C^{(-m-1)}_{k-1}\pmod{p^{N}}
\end{eqnarray*}
By applying the duality formula (\ref{duality formula}) to the both sides, 
we obtain the claim.
\end{proof}

%%%%%%%%%%%%%%%%%%%%%%%%%%%%%%%%%%%
\section{Congruence properties}

In this section, we discuss the congruence properties of poly-Bernoulli numbers with negative upper-indices. 
First, we show the following relation. 
\begin{theorem}\label{yosou2}
For any odd prime $p$ and any non-negative integer $k$, we have 
\begin{eqnarray*}
\B{-k}{p-1}\equiv 
\begin{cases}
1\pmod{p} & \text{if $k=0$ {\rm or} $k\not\equiv 0\pmod{p-1}$, } \\
2\pmod{p} & \text{if $k\neq 0$ {\rm and} $k\equiv 0\pmod{p-1}$. }
\end{cases}
\end{eqnarray*}
\end{theorem}

\begin{proof}
In case when $k=0$, it is easy to see $\B{0}{p-1}=\B{1-p}{0}=1$.

In case when $k>0$, we have  
\begin{eqnarray*}
\B{-k}{p-1}&=&(-1)^{p-1}\sum^{p-1}_{m=0}(-1)^{m}m!\st{p-1}{m}(m+1)^{k}\\
                     &\equiv &\sum^{p-2}_{m=1}(-1)^{m}m!\st{p-1}{m}(m+1)^{k} \pmod{p}
\end{eqnarray*}
using the explicit formula (\ref{explicit formula}). 
Applying the defining formula (\ref{stirling explicit}) of Stirling numbers of the second kind, 
it turns to
\begin{eqnarray*}
\B{-k}{p-1}&\equiv &\sum^{p-2}_{m=1}(m+1)^{k}\sum^{m}_{l=1}(-1)^{l}\binom{m}{l}l^{p-1} \pmod{p}. 
\end{eqnarray*}
Fermat's little theorem leads us to 
\begin{eqnarray*}
\B{-k}{p-1}&\equiv &\sum^{p-2}_{m=1}(m+1)^{k}\sum^{m}_{l=1}(-1)^{l}\binom{m}{l} \pmod{p}\\
                    &\equiv &\sum^{p-2}_{m=1}(m+1)^{k}\left( \sum^{m}_{l=0}(-1)^{l}\binom{m}{l} -1\right) \pmod{p}\\
                    &\equiv &\sum^{p-2}_{m=1}(m+1)^{k}\left( (1-1)^{m}-1\right) \pmod{p}\\
                    &\equiv &-\sum^{p-2}_{m=1}(m+1)^{k} \pmod{p}\\
                    &\equiv &
\begin{cases}
1\pmod{p} & \text{if $k\not\equiv 0\pmod{p-1}$, }\\
2\pmod{p} & \text{if $k\equiv 0\pmod{p-1}$. }
\end{cases}
\end{eqnarray*}
\end{proof}

The corresponding property of the above theorem for poly-Bernoulli numbers of type-$C$ is as follows.
\begin{theorem}\label{yosou2''}
For any odd prime $p$ and any non-negative integer $k$, we have 
\begin{eqnarray*}
C^{(-k-1)}_{p-2}\equiv 
\begin{cases}
0\pmod{p} & \text{if $k\not\equiv 0\pmod{p-1}$, } \\
1\pmod{p} & \text{if $k\equiv 0\pmod{p-1}$. }
\end{cases}
\end{eqnarray*}
\end{theorem}

\begin{proof}
By the explicit formula (\ref{explicit formula2}) for poly-Bernoulli numbers and 
the explicit formula (\ref{stirling explicit}) for Stirling numbers of the second kind, 
\begin{eqnarray*}
C^{(-k-1)}_{p-2}
&=& (-1)^{p-2}\sum^{p-2}_{m=0}(-1)^{m}m!\st{p-1}{m+1}(m+1)^{k+1}\\
&=& \sum^{p-2}_{m=0} (m+1)^{k} \sum^{m+1}_{l=0} (-1)^{l}\binom{m+1}{l} l^{p-1}
\end{eqnarray*}
Fermat's little theorem leads to 
\begin{eqnarray*}
C^{(-k-1)}_{p-2}&\equiv &\sum^{p-2}_{m=0}(m+1)^{k}\sum^{m+1}_{l=1}(-1)^{l}\binom{m+1}{l}\pmod{p}\\
                    &\equiv &\sum^{p-2}_{m=0}(m+1)^{k}\left\{ (1-1)^{m+1}-1\right\}\pmod{p}\\
                    &\equiv &-\sum^{p-2}_{m=0}(m+1)^{k}\pmod{p}\\
                    &\equiv &
\begin{cases}
0 \pmod{p} & \text{if $k\not\equiv 0 \bmod{p-1}$, }\\
1 \pmod{p} & \text{if $k\equiv 0 \bmod{p-1}$. }
\end{cases}
\end{eqnarray*}
\end{proof}

Furthermore Fermat's little theorem leads us to the following property. 
\begin{theorem}\label{yosou2'}
For any odd prime $p$ and any non-negative integer $k$, we have
\begin{eqnarray*}
C^{(-k-1)}_{p-1}\equiv 1\pmod{p}. 
\end{eqnarray*}
\end{theorem}

\begin{proof}
By the explicit formula (\ref{explicit formula2}), 
\begin{eqnarray*}
C^{(-k-1)}_{p-1}&=&(-1)^{p-1}\sum^{p-1}_{m=0}(-1)^{m}m!\st{p}{m+1}(m+1)^{k+1}\\
                     &\equiv &\sum^{p-2}_{m=0}(-1)^{m}m!\st{p}{m+1}(m+1)^{k+1} \pmod{p}. 
\end{eqnarray*}
Using the defining formula (\ref{stirling explicit}) of Stirling numbers of the second kind, 
\begin{eqnarray*}
C^{(-k-1)}_{p-1}&\equiv &-\sum^{p-2}_{m=0}(m+1)^{k}\sum^{m+1}_{l=1}(-1)^{l}\binom{m+1}{l}l^{p} \pmod{p}\\
&\equiv &-\sum^{p-2}_{m=0}(m+1)^{k+1}\sum^{m+1}_{l=1}(-1)^{l}\binom{m}{l-1}l^{p-1} \pmod{p}
\end{eqnarray*}
Fermat's little theorem leads to 
\begin{eqnarray*}
C^{(-k-1)}_{p-1}&\equiv &-\sum^{p-2}_{m=0}(m+1)^{k+1}\sum^{m+1}_{l=1}(-1)^{l}\binom{m}{l-1} \pmod{p}\\
                    &\equiv &-\sum^{p-2}_{m=0}(m+1)^{k+1}\sum^{m}_{l=0}(-1)^{l+1}\binom{m}{l}\pmod{p}\\
                    &\equiv &\sum^{p-2}_{m=0}(m+1)^{k+1}(1-1)^{m}\pmod{p}\\
                    &\equiv &1 \pmod{p}. 
\end{eqnarray*}
\end{proof}

The duality formula (\ref{duality formula}) for poly-Bernoulli numbers leads to the following. 
\begin{theorem}\label{yosou3}
Let $p$ be an odd prime and $k$ and $n$ be positive integers. 
If $k\equiv n\equiv 0, 1\pmod{p-1}$, then we have 
\begin{eqnarray*}
\B{-k}{n}\equiv 2\pmod{p}. 
\end{eqnarray*}
\end{theorem}

\begin{proof}
In case when $n\equiv 0\pmod{p-1}$, it holds by 
Theorem \ref{yosou2} that $\B{-n}{p-1}\equiv 2 \pmod{p}$, thus we obtain $\B{1-p}{p-1}\equiv 2 \pmod{p}$. 
In case when $n\equiv 1\pmod{p-1}$, by the explicit formula (\ref{explicit formula}), 
\begin{eqnarray*}
\B{-1}{1}&=&(-1)^{1}\sum^{1}_{m=0}(-1)^{m}m!\st{1}{m}(m+1)^{1}\\
                     &=&-(-1)\cdot 2=2. 
\end{eqnarray*}
Using the duality (\ref{duality formula}) and the periodicity (Theorem \ref{period1}), we obtain the claim in both cases. 
\end{proof}

Next we obtain the following property using the congruence relation between poly-Bernoulli numbers and classical Bernoulli numbers.
\begin{theorem}\label{yosou3'}
If $p$ is a prime number with $p\geq 7$, then 
\begin{eqnarray*}
\B{-p+3}{p-3}\equiv 0\pmod{p}. 
\end{eqnarray*}
\end{theorem}

\begin{proof}
We have $\B{-p+3}{p-3}=C^{(-p+3)}_{p-3}+C^{(-p+2)}_{p-4}$ by the definition. 
Using the duality formula (\ref{duality formula2}), we obtain $\B{-p+3}{p-3}=2C^{(-p+2)}_{p-4}$. 
Fermat's little theorem leads to $C^{(-p+2)}_{n}\equiv C_{n}\pmod{p}$ for $1\leq n\leq p-2$. 
Then we have $\B{-p+3}{p-3}\equiv 2C_{p-4}\pmod{p}$. 
By the given condition $p\geq 7$, $p-4 \geq 3$ is odd, so we have  
$C_{p-4}=0$ from the well known property of classical Bernoulli numbers. 
Thus we obtain $\B{-p+3}{p-3}\equiv 0\pmod{p}$.
\end{proof}

Adding up the poly-Bernoulli numbers of one period, 
we obtain the following results. 
\begin{theorem}\label{yosou4}
Let $p$ be a prime, and $k, N$ be positive integers. 
For any positive integer $n\geq N$, we have 
\begin{eqnarray*}
\sum^{\varphi(p^{N})-1}_{i=0}\B{-k}{n+i}\equiv
\sum^{\varphi(p^{N})-1}_{i=0}\B{-n-i}{k}\equiv 0\pmod{p^{N}}. 
\end{eqnarray*}
\end{theorem}

\begin{proof}
If $n=N$, using the explicit formula (\ref{explicit formula}) and the duality (\ref{duality formula}), we compute as
\begin{eqnarray*}
\sum^{\varphi(p^{N})-1}_{i=0}\B{-k}{N+i}
&=&\sum^{\varphi(p^{N})-1}_{i=0}\B{-N-i}{k}\\
&=&\sum^{\varphi(p^{N})-1}_{i=0}(-1)^{k}\sum^{k}_{m=0}(-1)^{m}m!\st{k}{m}(m+1)^{N+i}\\
&=&(-1)^{k}\sum^{k}_{m=1}(-1)^{m}m!\st{k}{m}\sum^{\varphi(p^{N})-1}_{i=0}(m+1)^{N+i}\\
&=&(-1)^{k}\sum^{k}_{m=1}(-1)^{m}m!\st{k}{m}\frac{(m+1)^{N}\{ (m+1)^{\varphi(p^{N})}-1 \}}{m}. 
\end{eqnarray*}
Euler's totient theorem leads us to obtain
\begin{eqnarray*}
\sum^{\varphi(p^{N})-1}_{i=0}\B{-k}{N+i}\equiv 0\pmod{p^{N}}. 
\end{eqnarray*}
The case $n>N$ reduces to the case $n=N$, by applying the periodicity given in Theorem \ref{period1}.
\end{proof}

By Theorem \ref{period1} and Theorem \ref{yosou4}, 
we generalize the results for modulo natural numbers. 

\begin{corollary}\label{cor3.7}
Let $k$ and $M$ be positive integers, and $M=p_{1}^{e_{1}}p_{2}^{e_{2}} \cdots p_{l}^{e_{l}}$ 
be the prime factorization of $M$, where $p_{1}, p_{2}, \ldots , p_{l}$ are distinct prime divisors, of $M$. 
For any integer $n \geq \max \{e_{j}\;|\;1\leq j \leq l\}$, we have 
\begin{eqnarray*}
\sum^{\varphi (M)-1}_{i=0}\B{-k}{n+i}\equiv
\sum^{\varphi (M)-1}_{i=0}\B{-n-i}{k}\equiv 0\pmod{M}. 
\end{eqnarray*}
\end{corollary}

\begin{proof}
By Theorem \ref{period1}, 
\begin{eqnarray*}
\sum^{\varphi (M)-1}_{i=0}\B{-k}{n+i}
&=&\sum^{\varphi (p_{1}^{e_{1}})\cdots \varphi (p_{l}^{e_{l}})-1}_{i=0}\B{-k}{n+i}\\
&\equiv &\prod_{\substack{j=1\\j \not= r}}^{l} \varphi (p_{j}^{e_{j}})\sum^{\varphi (p_{r}^{e_{r}})-1}_{i=0}\B{-k}{n+i} \pmod{p_{r}^{e_{r}}}
\end{eqnarray*}
for $1\leq r\leq l$. By Theorem \ref{yosou4}, the right-hand side is congruent to 0 $\bmod{\ p_{r}^{e_{r}}}$. 
Using the duality (\ref{duality formula}), we obtain the claim. 
\end{proof}

Tying up Brewbaker's relation (\ref{Brew}) and Corollary \ref{cor3.7}, we immediately obtain the following interesting property on the number of lonesome matrices.  
\begin{theorem}
Let $k$ and $M$ be positive integers, and $M=p_{1}^{e_{1}}p_{2}^{e_{2}} \cdots p_{l}^{e_{l}}$ 
be the prime factorization of $M$, where $p_{1}, p_{2}, \ldots , p_{l}$ are distinct prime divisors, of $M$. 
For any integer $n \geq \max \{e_{j}\;|\;1\leq j \leq l\}$, we have 
\[ \sum^{\varphi (M)-1}_{i=0}L(k, n+i)\equiv \sum^{\varphi (M)-1}_{i=0}L(n+i, k )\equiv 0\pmod{M}. \]
\end{theorem}

\section*{Acknowledgements}
This work is supported by Japan Society for the Promotion of Science, 
Grant-in-Aid for Scientific Research (C) 15K04774, 19K03437, 23K03026 (Y.O.).

\ 

\begin{flushleft}
\begin{small}
{Yasuo~Ohno}: 
{Mathematical Institute, Tohoku University, 
Aramaki Aza-Aoba 6-3, Aoba-ku, Sendai 980-8578, Japan}

e-mail: {\tt ohno.y@tohoku.ac.jp}

\ 

{Mika~Sakata}: 
{Department of Sport Sciences, School of Sport Sciences, 
Osaka University of Health and Sport Sciences, 
Asashirodai 1-1, Kumatori-cho, Sennan-gun, Osaka 590-0496, Japan}

e-mail: {\tt m.sakata@ouhs.ac.jp}
\end{small}
\end{flushleft}

\end{document}